\documentclass [a4,10pt,twoside]{article}
\usepackage{amsmath, amssymb, amsthm, amsfonts, amsxtra, latexsym, amscd,
pb-diagram,  graphics}
\newtheorem{dl}{Theorem}
\newtheorem{bd}[dl]{Lemma}
\newtheorem{nx}[dl]{Remark}

\newtheorem{md}[dl]{Proposition}

\newtheorem{hq}[dl]{Corollary}
\newtheorem*{dn}{Definition}

\newcommand{\tx}{\otimes }
\newcommand{\ts}{\oplus}
\newcommand{\Fa}{\breve{F}}
\newcommand{\Fn}{\widetilde{F}}

\newcommand{\A}{\mathcal{A}}
\newcommand{\B}{\mathcal{B}}
\newcommand{\I}{\mathcal{I}}
\newcommand{\R}{\mathcal{R}}
\newcommand{\Lh}{\mathcal{L}}
\newcommand{\Z}{\mathbb{Z}}

\begin{document}
\title{RING EXTENSION PROBLEM, SHUKLA COHOMOLOGY AND ANN-CATEGORY THEORY}
\author{Nguyen Tien Quang and Nguyen Thu Thuy}
\pagestyle{myheadings} 
\markboth{Ring extension problem, Shukla cohomology and Ann-category theory}{Nguyen Tien Quang and Nguyen Thu Thuy}
\maketitle
\setcounter{tocdepth}{1}
{\bf Abstract.} Every ring extension of $A$ by $R$ induces a pair of group homomorphisms $\mathcal{L}^{*}:R\to End_\Z(A)/L(A);\mathcal{R}^{*}:R\to End_\Z(A)/R(A),$ preserving multiplication, satisfying some certain conditions. A such $4$-tuple $(R,A,\mathcal{L}^{*},\mathcal{R}^{*})$ is called a ring pre-extension. Each ring pre-extension induces a $R$-bimodule structure on bicenter $K_A$ of ring $A,$ and induces an obstruction $k,$ which is a $3$-cocycle of $\Z$-algebra $R,$ with coefficients in $R$-bimodule $K_A$ in the sense of Shukla. Each obstruction $k$ in this sense induces a structure of a regular Ann-category of type $(R,K_A).$ This result gives us the first application of Ann-category in extension problems of algebraic structures, as well as in cohomology theories.
\section{Introduction}
Group extension problem has been presented with group of automorphisms $Aut(G)$ and quotient group $Aut(G)/In(G)$ by group of inner automorphisms. For ring extension problem, Mac Lane [1] has replaced the above groups by the ring of bimultiplications $M_A$ and the quotient ring $P_A$ of $M_A$ upon ring of inner bimultiplications. Besides, Maclane has replaced commutative, associative laws for addition in ring $R$ by commutative-associative law $(u+ v)+(r+s)=(u+r)+(v+s)$ and therefore proved that each obstruction of ring extension problem is an element of $3$-dimensional cohomology group in the sense of Maclane, and the number of solutions corresponds to $2$-dimensional cohomology group of ring under a bijection.\\
The idea of solving group extension problem by groups $Aut(G)$ and $Aut(G)/In(G)$ can be applied to ring extension theory in a different way. In this way, we use separately associative, commutative laws to construct obstruction of ring extension problem and therefore give the solution to ring extension problem in the terms of Shukla cohomology [5] of $\Z$-algebras. This cohomology, in our opinion, is more convenient than cubical resolution in Maclane[1].\\
$\Z$-split ring extension problem relates close to Hochschild cohomology and is regarded as an application of Ann-category theory. In this paper, we establish the relationship between ring extension problem in the general case and Ann-category theory. In this way, we may use Ann-category as a united terms to interpret cohomology of different algebraic systems.
\section{Cohomology of an associative algebra}
Shukla cohomology (see [5]) of ring $R$ (regarded as a $\Z$-algebra) with coefficients in $R$-bimodule $M$ is the group $$H^*_S(R,M)=H^*(\sum_{n\geq 0} Hom_{\Z}(U^n,M)),$$ 
where $U$ is a graded differential $\Z$-algebra as well as a free resolusion of $R$ on $\Z.$

To facilitate the calculation, in [2], N.T.Quang has built nomalized complex $B(U) = \sum U\tx(U/\Z)^n\tx U,$ in which $(U/\Z)^0=\Z, U/\Z=U/{I\Z},$ where $I:\Z\longrightarrow U$ is a canonical homomorphism. Each generator of $B(U)$ has the form
$$x=u_0[u_1|...u_n]u_{n+1}, n\geq 0$$
This element is equal to $0$ if there exists one in $u_i (i = 1,...,n)$ belonging to $I\Z.$ Then, the normalized complex $B(U)$ is a graded differential bimodule on $DG$-algebra $U$ with grading is defined by
$$deg(u_0[u_1|...|u_n]u_{n+1})=n+degu_0 + ... +degu_{n+1}$$
and with differential $\partial=\partial_r+\partial_s$ is defined by
$$\partial_r(u_0[u_1|...|u_n]u_{n+1})=du_0[u_1|...|u_n]u_{n+1}$$
$$-\sum_{i=1}^n(-1)^{e_{i-1}}u_0[u_1|...|du_i|...|u_n]u_{n+1}+(-1)^{e_n}u_0[u_1|...|u_n](du_{n+1})$$
$$\partial_s(u_0[u_1|...| u_n]u_{n+1})=(-1)^{e_0}u_0u_1[u_2|...|u_n]u_{n+1}$$
$$+\sum_{i=1}^{n-1}(-1)^{e_i}u_0[u_1|...|u_iu_{i+1}|...|u_n]u_{n+1}+(-1)^{e_n}u_0[u_1|...|u_{n-1}]u_nu_{n+1}$$
where $e_0=0,\ e_i=i+degu_0+\cdots +degu_i.$
\section{Classical problem: Singular ring extension}
A \emph{ring extension} is a ring epimorphism $\sigma:S \longrightarrow R$ which carries the identity of $S$ to the identity of $R$. Then, $A=Ker\sigma$ is a two-sided ideal in $S$ and therefore, we have the short exact sequence of rings and ring homomorphisms
\[
\begin{diagram}\node{E:\quad 0} \arrow{e,t}{}
\node{A} \arrow{e,t}{\chi}
\node{S}\arrow{e,t}{\sigma} 
\node{R} \arrow{e,t}{} 
\node{0\quad}\node{,\quad\sigma(1_S)=1_R.}
\end{diagram}
\]
Extension $E$ is called \emph{singular} if $A$ is a ring with \emph{null multiplication}, i.e., $A^2 = 0$. Then, $A$ becomes a $R$-bimodule with operators
$$xa = u(x)a; ax = au(x), a\in A, x\in R.$$
where $u(x)$ is a representative of $R$ in $S$ in the sense $\sigma u(x) = x$ (note that we always choose $u(0)=0, u(1_R)=1_S$).

Let $E$ be a given singular extension of $A$ by $R$ and $u(x)$ be one of its representatives. Then, addition and multiplication in $S$ induce two factor sets $f,g$ determined by
\begin{equation}
\begin{split}
\begin{aligned}
u(x)+u(y)&= f(x,y)+u(x+y)\\
u(x).u(y)&= g(x,y)+u(xy)
\end{aligned}
\end{split}
\end{equation}
where $f, g: R^2 \longrightarrow A$\\
Since $u(0)=0,u(1_R)=1_S,$ $f$ and $g$ satisfy normalization condition in the sense
\begin{equation}
\begin{split}
\begin{aligned}
f(x,0)&= f(0,y)=0\\
g(x,0)&= g(0,y)=g(1,y)=g(y,1)=0
\end{aligned}
\end{split}
\end{equation}
From commutative, associative laws for addition in $S,$ respectively, we have 
\begin{equation}
\begin{split}
f(x,y)=f(y,x)
\end{split}
\end{equation}
\begin{equation}
\begin{split}
f(y,z)-f(x+y,z)+f(x,y+z)-f(x,y)=0
\end{split}
\end{equation}
From associative law for multiplication in $S,$ we have
\begin{equation}
\begin{split}
xg(y,z)-g(xy,z)+g(x,yz)-g(x,y)z=0
\end{split}
\end{equation}
From left and right distributive laws for addition and multiplication, respectively, we have 
\begin{equation}
\begin{split}
\begin{aligned}
xf(y,z)-f(xy,xz)&=g(x,y)+g(x,z)-g(x,y+z)\\
(f(x,y))z-f(xz,yz)&=g(x,z)+g(y,z)-g(x+y,z)
\end{aligned}
\end{split}
\end{equation}
The pair $(f, g)$ satisfies relations (2)-(6) is called \emph{a factor set} of extension $E$.
\begin{md}
A factor set of a singular extension of $A$ by $R$ is a $2$-cocycle of ring $R$ with coefficients in $R$-module $A$ in the sense of Mac Lane-Shukla.
\end{md}
\begin{proof}
This result is obtained from calculation of group of $2$-cocycles of ring $R$ with coefficients in $R$-bimodule $A.$
\end{proof}
Calculation of group $H^2(R,A)$ based on definitions of Maclane as well as of Shukla coincide with each other. Differences in representation as well as complexity in calculation of these two cohomologies occur when the dimension is down to $3.$

If we choose factor set $u'(x)$ instead of $u(x)$ such that $u'(0)=0,u'(1)=1,$ we will have $\sigma(u'(x)-u(x))=0,$ so $u'(x)=u(x)+t(x),$ where $t(x)\in A.$ Then we have
\begin{equation}
\begin{split}
\begin{aligned}
f'(x,y)&= f(x,y)-t(x+y)+t(x)+t(y)=f(x,y)+(\delta_1t)(x,y)\\
g'(x,y)&= g(x,y)+xt(y)-t(xy)+t(x)y=g(x,y)+(\delta_2t)(x,y)
\end{aligned}
\end{split}
\end{equation}
This shows that each singular ring extension $E$ responds to a factor set $(f', g')$ which is equal to $(f, g)$ up to a $2$-coboundary and therefore each singular extension $E$ of $A$ by $R$ responds to an element of cohomology group $H^2(R,A)$ in the sense of Mac Lane - Shukla.

Conversely, if we have $R$-bimodule $A$ together with functions $f,g$ satisfying relations (2)-(6), we can construct extension $E$ of $A$ by $R,$ where ring $S$ is determined by
\begin{eqnarray}
S&=&\{(a,x)\mid a\in A,x\in R\}\nonumber
\end{eqnarray}
together with two operations 
\begin{equation}
\begin{split}
\begin{aligned}
(a,x)+(b,y)&=&(a+b+f(x,y),x+y)\\
(a,x)(b,y)&=&(ay+xb+g(x,y),xy)
\end{aligned}
\end{split}
\end{equation}
\section{Pre-extensions of rings}
We now consider the general case, i.e., ring extension
\[
\begin{diagram}\node{E:\quad 0} \arrow{e,t}{}
\node{A} \arrow{e,t}{\chi}
\node{S}\arrow{e,t}{\sigma} 
\node{R} \arrow{e,t}{} 
\node{0\quad}\node{,\quad\sigma(1_S)=1_R.}
\end{diagram}
\]
which is unnecessarily singular. Then $A$ is a two-sided ideal of ring $S$ and therefore $A$ is regarded as a $S$-bimodule and $\chi$ is a $S$-bimodule homomorphism.

To determine necessary conditions for ring estension problem, we will expand the technology which is used for group extension problem. This way is a bit different from the one of Maclane in [1]. In this way, we establish directly the relationship between ring extension problem, Shukla cohomology and Ann-category theory.

We consider the following two sets of mappings instead of the set of bimultiplications of ring $A$ which Maclane as well as Shukla have done
\begin{eqnarray}
L(A)&=&\{l_a: A\longrightarrow A,a\in A\mid l_a(b)=ab,\forall b\in A\}\nonumber\\
R(A)&=&\{r_a:A\longrightarrow A,a\in A\mid r_a(b)=ba,\forall b\in A\}\nonumber
\end{eqnarray}
Clearly, each $l_a$ (resp. $r_a$) is a right (resp. left) $A$-module endomorphism of $A$ and $L(A), R(A)$ are two subrings of ring $End_{\Z}(A)$ of endomorphisms of group $A.$\\
Consider ring homomorphism
\begin{eqnarray}
\mathcal{L}:\quad S&\rightarrow&End_{\Z}(A)\nonumber\\
s&\mapsto&\mathcal{L}_s,\mathcal{L}_s(a)=sa,\forall a\in A.\nonumber
\end{eqnarray}
Clearly, $\mathcal{L}_s$ is a right $A$-module endomorphism of $A,$ and $\mathcal{L}$ induces group homomorphism $\mathcal{L}^{*}:R\to End_{\Z}(A)/L(A)$ such that following diagram
\[
\begin{diagram}
\node{0} \arrow[1]{e,t}{}
\node{A} \arrow[1]{e,t}{\chi} \arrow[1]{s,r}{l} 
\node{S} \arrow{e,t}{\sigma}\arrow{s,r}{\mathcal{L}}
\node{R}  \arrow{e,t}{}\arrow{s,r}{\mathcal{L}^{*}}\node{0}\\
\node{0} \arrow{e,t}{}
\node{L(A)} \arrow{e,t}{i} 
\node{End_{\Z}(A)} \arrow{e,t}{p}
\node{End_{\Z}(A)/L(A)} \arrow{e,t}{}
\node{0}
\end{diagram}
\]
commute.\\
Similarly, we have ring homomorphism
\begin{eqnarray}
\mathcal{R}:\quad S&\rightarrow&End_{\Z}(A)\nonumber\\
s&\mapsto&\mathcal{R}_s,\mathcal{R}_s(a)=as,\forall a\in A.\nonumber
\end{eqnarray}
$\mathcal{R}_s$ is a left $A$-module endomorphism of $A,$ and $\mathcal{R}$ induces group homomorphism $\mathcal{R}^{*}:R\to End_{\Z}(A)/R(A)$ such that following diagram
\[
\begin{diagram}
\node{0} \arrow[1]{e,t}{}
\node{A} \arrow[1]{e,t}{\chi} \arrow[1]{s,r}{r} 
\node{S} \arrow{e,t}{\sigma}\arrow{s,r}{\mathcal{R}}
\node{R}  \arrow{e,t}{}\arrow{s,r}{\mathcal{R}^{*}}\node{0}\\
\node{0} \arrow{e,t}{}
\node{R(A)} \arrow{e,t}{i} 
\node{End_{\Z}(A)} \arrow{e,t}{p}
\node{End_{\Z}(A)/R(A)} \arrow{e,t}{}
\node{0}
\end{diagram}
\]
commute.\\
Moreover, $\mathcal{L}_s$ and $\mathcal{R}_s$ satisfy relations
\begin{equation}
\begin{split}
l_a\circ \mathcal{L}_s=l_{\mathcal{R}_s(a)}\qquad,\mathcal{L}_s\circ l_a=l_{\mathcal{L}_s(a)}\\
\mathcal{R}_s\circ r_a=r_{\mathcal{R}_s(a)}\qquad,r_a\circ
\mathcal{R}_s=r_{\mathcal{L}_s(a)}
\end{split}
\end{equation}
$$\mathcal{L}_s\circ\mathcal{R}_{s'}=\mathcal{R}_{s'}\circ\mathcal{L}_s\nonumber$$
forall $a\in A$ and $s,s'\in S.$\\
From relations (9), we can see that $L(A)$ and $R(A)$ are, respectively, two-sided ideals in $\mathcal{L}(S)$ and $\mathcal{R}(S)$. Therefore, $\mathcal{L}^{*}(R)$ and $\mathcal{R}^{*}(R)$ are rings, $\mathcal{L}^{*}$ and $\mathcal{R}^{*}$ are ring homomorphisms from $R$ to its images.

Let $A$ be a ring (without identity) and $R$ be a ring with identity $1\ne 0.$ Assume that we have group homomorphisms
$$\mathcal{L}^{*}:R\to End_{\Z}(A)/L(A);\mathcal{R}^{*}:R\to End_{\Z}(A)/R(A)$$
such that for each $x\in R,$ there exists a pair $\varphi_x\in \mathcal{L}^{*}(x),\psi_x\in\mathcal{R}^{*}(x)$ satisfying relations
\begin{equation}
\begin{split}
\phi_1=\psi_1=id\qquad\qquad\qquad\qquad\\
\begin{matrix}
 l_a\circ \varphi_x&=&l_{\psi_x(a)}\qquad,\qquad
\varphi_x\circ l_a&=&
l_{\varphi_x(a)}\\
\psi_x\circ r_a&=&r_{\psi_x(a)}\qquad,\qquad r_a\circ \psi_x&=&
r_{\varphi_x(a)}
\end{matrix}\\
\varphi_x\circ\psi_y=\psi_y\circ\varphi_x\qquad\qquad\qquad
\end{split}
\end{equation}
Moreover, $\mathcal{L}^{*}$ and $\mathcal{R}^{*}$ preserve multiplication. Then, we call $4$-tuple $(R,A,\mathcal{L}^{*},\mathcal{R}^{*})$ \emph{a pre-extension} of $A$ by $R$ inducing $\mathcal{L}^{*},\mathcal{R}^{*}.$\\
Ring extension problem is to find whether pre-extension $(R,A,\mathcal{L}^{*},\mathcal{R}^{*})$ has extension and, if so, how many extensions of $A$ by $R$ are.
\section{The obstruction of a ring pre-extension}
We now present the concept of \emph{obstruction} of pre-extension $(R,A,\mathcal{L}^{*},\mathcal{R}^{*}).$\\
Since $\mathcal{L}^{*}(x).\mathcal{L}^{*}(y)=\mathcal{L}^{*}(xy),$ we have 
$$\varphi_x.\varphi_y=\varphi_{xy}+l_{g(x,y)}\qquad,\qquad g(x,y)\in A$$
Then, from associative law for multiplication in $End_{\Z}(A),$ we have
$$l_{\varphi_x[g(y,z)]}+l_{g(xy,z)}=l_{g(xy,z)}+l_{\psi_z[g(x,y)]}$$
which yeilds
\begin{equation}
\begin{split}
\varphi_x[g(y,z)]-g(xy,z)+g(x,yz)-\psi_z[g(x,y)]=\alpha(x,y,z),
\end{split}
\end{equation}
in which $\alpha(x,y,z)\in K_A,$ where $K_A$ is a two-sided ideal of $A.$
$$K_A=\{c\in A\mid ca=0=ac \ \ \forall a\in A\}.$$
We call $K_A$ \emph{bicenter} of $A.$
Associative, commutative laws for addition in $End_{\Z}(A),$ respectively, give us
\begin{equation}
\begin{split}
\begin{matrix}
f(y,z)-f(x+y,z)+f(x,y+z)-f(x,y)&=&\xi(x,y,z)\\
f(x,y)-f(y,x)&=&\eta(x,y)
\end{matrix}
\end{split}
\end{equation}
where $\xi:R^3\to K_A,\ \eta:R^2\to K_A.$\\
Finally, from distributive law in $End_{\Z}(A),$ we have
\begin{equation}
\begin{split}
\varphi_x[f(y,z)]-f(xy,xz)+g(x,y+z)-g(x,y)-g(x,z)=\lambda(x,y,z)\\
\psi_z[f(x,y)]-f(xz,yz)+g(x+y,z)-g(x,z)-g(y,z)=\rho(x,y,z)
\end{split}
\end{equation}
where $\lambda,\rho:R^3\to K_A.$\\
We call the family $(\xi,\eta,\alpha,\lambda,\rho)$ which is so determined \emph{an obstruction} of pre-extension $(R,A,\mathcal{L}^{*},\mathcal{R}^{*}).$\\
Clearly, $K_A$ is a $R$-bimodule with operations $xa=\varphi_x(a),ax=\psi_x(a),$ which are independent of the choice of $\varphi_x,\psi_x.$
\begin{nx}
We may note that if commutative-associative law
$$(u+v)+(r+s)=(u+r)+(v+s)$$
for addition in $End_{\Z}(A)$ is used, we will have a function $\gamma:R^4\to K_A,$ given by
$$\gamma(x,y,z,t)=f(x+y,z+t)-f(x,z)-f(y,t)-f(x+z,y+t)+f(x,y)+f(z,t)\quad(12')$$
Then, an obstruction of pre-extension $(R,A,\mathcal{L}^{*},\mathcal{R}^{*})$ is a family of 4 functions $(\alpha,\gamma,\lambda,\rho)$ satisfying relations (11), (12') and (13). We can see that two concepts of above-mentioned obstruction are equivalent. Due to the theory of cohomology of Maclane, each obstruction $(\alpha,\gamma,\lambda,\rho)$ is a $3$-cocycle in $Z^3_M(R,K_A).$
\end{nx}
One of the main results in this paper is showing that each obstruction regarded as a family of 5 functions $(\xi,\eta,\alpha,\lambda,\rho)$ is a $3$-cocycle in the sense of Shukla.

First, we may prove two following lemmas for obstructions.

\begin{bd}
If we fix $\varphi_x\in\mathcal{L}^{*}(x), \psi_x\in\mathcal{R}^{*}(x),$ and replace functions $g,f$ by functions $g',f'$ such that 
$$l_{g(x,y)}=l_{g'(x,y)}\ \ \textrm{and}\ \ l_{f(x,y)}=l_{f'(x,y)}$$
then $g'=g+\nu$ and $f'=f+\mu,$ where $\nu,\mu:R^2\to K_A.$
Then,
$$\xi'=\xi+\partial_1\mu,\qquad\eta'=\eta+ant\mu,\qquad\alpha'=\alpha+\partial_2\nu$$
$$(\partial_1\mu)(x,y,z)=\mu(y,z)-\mu(x+y,z)+\mu(x,y+z)-\mu(x,y)$$
\begin{equation}
\begin{split}
(\partial_2\nu)(x,y,z)=x\nu(y,z)-\nu(xy,z)+\nu(x,yz)-\nu(x,y)z
\end{split}
\end{equation}
$$\lambda'(x,y,z)=\lambda(x,y,z)+\nu(x,y+z)-\nu(x,y)-\nu(x,z)+x\mu(y,z)-\mu(xy,xz)$$
$$\rho'(x,y,z)=\rho(x,y,z)+\nu(x+y,z)-\nu(x,z)-\nu(y,z)+\mu(x,y)z-\mu(xz,yz)$$
Moreover, two functions $\nu,\mu$ can be chosen arbitrarily.
\end{bd}
\begin{bd}
If we replace functions $\varphi_x,\psi_x$ by functions $\varphi'_x,\psi'_x,$ we will be able to choose functions $g',f'$ such that family $(\xi,\eta,\alpha,\lambda,\rho)$ is unchanged.
\end{bd}
\indent Two obstructions $(\xi,\eta,\alpha,\lambda,\rho)$ and $(\xi',\eta',\alpha',\lambda',\rho')$ of the same pre-extension $(R,A,\mathcal{L}^{*},\mathcal{R}^{*})$ are \emph{cohomologous} if they satisfy relations (14) where $\nu,\mu$ are certain functions.

Two above lemmas show that each pre-extension $(R,A,\mathcal{L}^{*},\mathcal{R}^{*})$ determines uniquely the cohomology class of any one of its obstructions.
We now show the solution of ring extension problem by Shukla cohomology of ring (regarded as $\Z$-algebra)
\begin{dl}
The cohomology class of obstruction $(\xi,\eta,\alpha,\lambda,\rho)$ of pre-extension $(R,A,\mathcal{L}^{*},\mathcal{R}^{*})$ is an element $\overline{k}\in H^3_{Sh}(R,K_A).$ Pre-extension $(R,A,\mathcal{L}^{*},\mathcal{R}^{*})$ has extension iff $\overline{k}=0.$ Then, there exists a bijection between the set $Ext(R,A)$ of equivalence classes of extension and the set $H^3_{Sh}(R,K_A).$
\end{dl}
\begin{proof}
In [2], N.T.Quang has built acyclic and $\Z$-free complex as follows
$$U= \sum_{i=0}^4 U_i,$$
where $U_i$ are elements of the exact sequence
$$0\rightarrow U_4\stackrel{d_4}{\rightarrow}U_3\stackrel{d_3}{\rightarrow}U_2\stackrel{d_2}{\rightarrow}U_1\stackrel{d_1}{\rightarrow}U_0\stackrel{\epsilon}{\rightarrow}R\rightarrow 0$$
of free abelian groups\\
\[\begin{aligned}
U_0 &= \Z(R^0), R^0 = R\setminus\{ 0\}\\
U_1 &= \Z(R^0\times R^0)\\
U_2 &= \Z(R^0\times R^0\times R^0)\oplus U_1\\
U_3 &= \Z(R^0\times R^0\times R^0\times R^0)\oplus U_2\oplus U_0\\
U_4 &= Kerd_3
\end{aligned}\]\\
Homomorphisms $d_i$ and $\epsilon$ are given by\\
\[\begin{aligned}\epsilon[x]&=x\\
d_1[x,y]&=[y]-[x+y]+[x]\\
d_2[x,y,z]&=[y,z]-[x+y,z]+[x,y+z]-[x,y]\\
d_2[x,y]&=[x,y]-[y,x]\\
d_3[x,y,z,t]&=[y,z,t]-[x+y,z,t]+[x,y+z,t]-[x,y,z+t]+[x,y,z]\\
d_3[x,y,z]&= [x,y,z]-[x,z,y]+[z,x,y]-[y,z]+[x+y,z]-[x,z]\\
d_3[x,y]&=[x,y]+[y,x]\\
d_3[x]&=[x,x]
\end{aligned}\]\\
$d_4$ is a canonical embedding.\\
After that, we determine an associative multiplication (which satisfies Leibniz formula for differential $d$) to make $U$ become a DG-algebra on $\Z.$

1) First, for $[x]\in U_0$ we set
$$[x][x_1,\ldots,x_n]=[xx_1,\ldots,xx_n]$$
$$[x_1,\ldots,x_n][x]=[x_1x,\ldots,x_nx]$$
$$[x,y][z,t]=-[xz+yz]+[xz,yz,xt]+[xz+xt,yz,yt]-[xz,xt,yz]+[xt,yz]$$

2) In other cases, for $a_i,a_j$ which are, respectively, generators of $U_i,\ U_j$ we have
$$(da_i)a_j+(-1)^ia_i(da_j)\in Ker(d_{i+j-1})$$
so they vanish under operation of differentials $d.$ Besides, the short exact sequence of abelian groups
\[
\begin{diagram}
\node{0} \arrow{e,t}{}
\node{Ker(d_{i+j})} \arrow{e,t}{}
\node{U_{i+j}}\arrow{e,t}{d_{i+j}} 
\node{Ker(d_{i+j-1})} \arrow{e,t}{} 
\node{0}
\end{diagram}
\]
splits since $Ker(d_{i+j-1})$ is $\Z$-free. So there exists a group injection $s:Ker(d_{i+j-1})\to U_{i+j}$ such that $d_{i+j}\circ s=id.$ Then, we set
$$a_ia_j=s[(da_i)a_j+(-1)^ia_i(da_j)].$$
With the above-constructed $DG$-algebra $U,$ a $3$-cocycle in $Z^3_{Sh}(R,K_A)$ is an obstruction $(\xi,\eta,\alpha,\lambda,\rho)$ of pre-extension $(R,A,\mathcal{L}^{*},\mathcal{R}^{*}).$ 

Moreover, two $3$-cocycles $(\xi,\eta,\alpha,\lambda,\rho)$ and $(\xi,\eta,\alpha,-\lambda,\rho)$ belong to the same cohomology class of $H^3_{Sh}(R,K_A)$ iff they are cohomologous obstructions. So we deduce that pre-extension $(R,A,\mathcal{L}^{*},\mathcal{R}^{*})$ determine uniquely element $\overline{k}\in H^3_{Sh}(R,K_A).$

If $(R,A,\mathcal{L}^{*},\mathcal{R}^{*})$ has extension, obviously, $\overline{k}=0.$ Conversely, if $\overline{k}=0,$  $k=(\xi,\eta,\alpha,\lambda,\rho)$ where
$$\xi=\partial_1\mu,\qquad\eta=ant\mu,\qquad\alpha=\partial_2\nu$$
$$\lambda(x,y,z)=\nu(x,y+z)-\nu(x,y)-\nu(x,z)+x\mu(y,z)-\mu(xy,xz)$$
$$\rho(x,y,z)=\nu(x+y,z)-\nu(x,z)-\nu(y,z)+\mu(x,y)z-\mu(xz,yz).$$
Applying Lemma 4 by replacing $g,f$ by $g'=g-\nu,f'=f-\mu,$ we will have
$$k'=(\xi',\eta',\alpha',\lambda',\rho')=0$$
Then, we may construct extension $S$ of $A$ by $R$ as in relations (8), in which $f,g$ are replaced by $f',g'.$ The rest of the theorem is proved by the above-mentioned events, without new technology.
\end{proof}
\section{Ann-category and solution of ring extension problem}
\indent In this section, we show the relationship between Ann-category theory and ring extension problem. See [2], [3] for definitions and fundamental results of Ann-categories.
\begin{dn}
Let $R$ be a ring with identity $1\neq 0$ and $A$ be a $R$-bimodule. An Ann-category $\I$ of type $(R,A)$ is a category whose objects are elements of $R$ and whose morphisms are endomorphisms. In the concrete, for $r,s\in R,$ we define
$$Hom(r,s) = \emptyset \ \ \textrm{where} \ \ r\ne s$$
$$Hom(s,s) = Aut(s)=\{s\}\times A.$$
The composite of two morphisms is induced by addition in $A.$\\
Two tensor products $\tx,\ts$ on $\I$ are defined by
\begin{eqnarray}
r\ts s&=&r+s \ \ \textrm{(addition in ring $R$)}\nonumber\\
(r,u)\ts(s,v)&=&(r+s,u+v)\nonumber\\
r\tx s&=&rs \ \ \textrm{(product in ring $R$)}\nonumber\\
(r,u)\tx(s,v)&=&(rs,rv+us).\nonumber
\end{eqnarray}
Constraints on $\I$ are defined to be a family
$$(\xi,\eta,(0,id,id),\alpha,(1,id,id),\lambda,\rho)$$
where $\eta:R^2\to A,\ \alpha,\lambda,\rho:R^3\to A$ are functions so that the set of axioms of an Ann-category is satisfied.
\end{dn}

Ann-category of type $(R,A)$ is \emph{regular} if its commutivity constraint satisfies $\eta(x,x)=0$ for all $x\in R.$

The following propositions give us the first relationship between ring extension problem and Ann-category theory.
\begin{md}
Each obstruction $(\xi,\eta,\alpha,\lambda,\rho)$ of pre-extension $(R,A,\mathcal{L}^{*},\mathcal{R}^{*})$ is a structure of regular Ann-category $\I$ of type $(R,K_A).$
\end{md}
\begin{proof}
From relations (10)-(13), we may verify directly that family $(\xi,\eta,\alpha,\lambda,\rho)$ satisfies following relations
\begin{enumerate}
\item[1.]$\xi(y,z,t)-\xi(x+y,z,t)+\xi(x,y+z,t)-\xi(x,y,z+t)+\xi(x,y,z)=0$
\item[2.]$\xi(0,y,z)=\xi(x,0,z)=\xi(x,y,0)=0$
\item[3.]$\xi(x,y,z)-\xi(x,z,y)+\xi(z,x,y)-\eta(x,z)+\eta(x+y,z)-\eta(y,z)=0$
\item[4.]$\eta(x,y)+\eta(y,x)=0$
\item[5.]$\eta(x,x) = 0$
\item[6.]$x\eta(y,z)-\eta(xy,xz)=\lambda(x,y,z)-\lambda(x,z,y)$
\item[7.]$\eta(x,y)z-\eta(xz,yz)=\rho(x,y,z)-\rho(y,x,z)$
\item[8.]$x\xi(y,z,t)-\xi(xy,xz,xt)=\lambda(x,z,t)-\lambda(x,y+z,t)+\lambda(x,y,z+t)-\lambda(x,y,z)$ 
\item[9.]$\xi(x,y,z)t-\xi(xt,yt,zt)=\rho(y,z,t)-\rho(x+y,z,t)+\rho(x,y+z,t)-\rho(x,y,t)$
\item[10.]$\rho(x,y,z+t)-\rho(x,y,z)-\rho(x,y,t)+\lambda(x,z,t)
  +\lambda(y,z,t)-\lambda(x+y,z,t)\\=-\xi(xz+xt,yz,yt)
+\xi(xz,xt,yz)-\eta(xt,yz)+\xi(xz+yz,xt,yt)-\xi(xz,yz,xt)$
\item[11.]
  $\alpha(x,y,z+t)-\alpha(x,y,z)-\alpha(x,y,t)=
  x\lambda(y,z,t)+\lambda(x,yz,yt)-\lambda(xy,z,t)$
\item[12.]
  $\alpha(x,y+z,t)-\alpha(x,y,t)-\alpha(x,z,t)\\=
x\rho(y,z,t)-\rho(xy,xz,t)+\lambda(x,yt,zt)-\lambda(x,y,z)t$
\item[13.]
  $\alpha(x+y,z,t)-\alpha(x,z,t)-\alpha(y,z,t)=
-\rho(x,y,z)t-\rho(xz,yz,t)+\rho(x,y,zt)$
\item[14.]
  $x\alpha(y,z,t)-\alpha(xy,z,t)+\alpha(x,yz,t)
  -\alpha(x,y,zt)+\alpha(x,y,z)t=0$  
\item[15.]$\alpha(1,y,z)=\alpha(x,1,z)=\alpha(x,y,1)=0$
\item[16.]$\alpha(0,y,z)=\alpha(x,0,t)=\alpha(x,y,0)=0$
\item[17.]$\lambda(1,y,z)=\lambda(0,y,,z)=\lambda(x,0,z)=\lambda(x,y,0)=0$
\item[18.]$\rho(x,y,1)=\rho(0,y,z)=\rho(x,0,z)=\rho(x,y,0)=0$
\end{enumerate} 
where $x,y,z,t\in R.$

These relations (without the fifth relation) show that the family of constraints $(\xi,\eta,(0,id,id),\alpha,(1,id,id),\lambda,\rho)$ satisfies the set of axioms of an Ann-category. Then, it is said that family $(\xi,\eta,\alpha,\lambda,\rho)$ is a structure of Ann-category of type $(R,A).$

From the fifth relation, Ann-category $(R,A)$ is regular.
\end{proof}
\begin{md}
$[4]$ Family $(\xi,\eta,\alpha,\lambda,\rho)$ is a structure of regular $Ann$-category of type $(R,A)$ iff $(\xi,\eta,\alpha,-\lambda,\rho)$ is a $3$-cocycle of ring $R$ with coefficients in $R$-bimodule $A$ in the sense of Shukla cohomology. 
\end{md}
\begin{proof}
The proof is deduced from the proof of Theorem 5 and Proposition 6
\end{proof}
\begin{dn}
Let $\A,\B$ be Ann-categories. An Ann-functor from  $\A$ to $\B$ is a triple $(F,\Fa,\Fn)$ where $(F,\Fa)$ is a symmetric monoidal $\ts$-functor and $(F,\Fn)$ is a monoidal $\tx$-functor so that following diagrams
{\scriptsize 
\[
\divide \dgARROWLENGTH by 2
\begin{diagram}
\node{F(X(Y\ts Z))}\arrow{e,t}{\Fn}\arrow{s,l}{F(\Lh)}\node{FX F(Y\ts Z)}\arrow{e,t}{id \tx \Fa}\node{FX(FY\ts FZ)}\arrow{s,r}{\Lh'}\\
\node{F((XY)\ts (XZ))}\arrow{e,t}{\Fa}\node{F(X Y) \ts F(X Z)}\arrow{e,t}{\Fn\oplus\Fn}\node{(FX FY)\ts (FX FZ)}
\end{diagram}\]}
{\scriptsize 
\[
\divide \dgARROWLENGTH by 2
\begin{diagram}
\node{F((X\ts Y)\tx Z)}\arrow{e,t}{\Fn}\arrow{s,l}{F(\R)}
\node{F(X\ts Y) FZ}\arrow{e,t}{\Fa\tx id}\node{(FX\ts FY)FZ}\arrow{s,r}{\R'}\\
\node{F((X Z)\ts (Y Z))}\arrow{e,t}{\Fa}\node{F(X Z) \ts F(Y Z)}\arrow{e,t}{\Fn\oplus\Fn}\node{(FX FZ)\ts (FY FZ)}
\end{diagram}\]}
commute
\end{dn}
\begin{dn} Two structures $f =(\xi,\eta,\alpha,\lambda,\rho)$ and $f'=(\xi,\eta,\alpha,\lambda,\rho)$ of Ann-category of type $(R,A)$ are cohomologous iff there is an Ann-functor $(F,\Fa,\Fn): (R,A,f)\longrightarrow (R,A,f')$ where $F = id.$
\end{dn}

From the definition of Ann-functor, we have following proposition
\begin{md}
Two structures $f$ and $f'$ of Ann-category of type $(R,A)$ are cohomologous iff they satisfy relations mentioned in Lemma 3.
\end{md}
\indent Proposition 7 and 8 give us the following corollary.
\begin{hq}
Each cohomology class of structures of regular Ann-categories of type $(R,A)$ is an element in $H^3_{Sh}(R,K_A).$
\end{hq}
\begin{proof}
Indeed, if Ann-category of type $(R,A)$ is regular, $f$ and $f'$ are $3$-cocycles of group $H^3_{Sh}(R,K_A).$ Since they satisfy Lemma 3, they are equal up to a $3$-coboundary, and therefore it completes the proof.
\end{proof}

\vspace{12pt}

\begin{center}

\end{center}
Math. Dept., Hanoi University of Education\\
E-mail adresses: nguyenquang272002@gmail.com
\end{document}